\documentclass[11pt]{amsart}

\usepackage{amssymb,amscd,amsthm, verbatim,amsmath,color,fancyhdr, mathrsfs, bm}
\usepackage{graphicx}
\usepackage{turnstile}

\usepackage{tikz,tikz-cd}

\usepackage{array}
\usepackage{multirow}

\usepackage{amssymb,epsfig,psfrag}
\usepackage{amsfonts, amssymb}
\usepackage{enumerate}
\usepackage[all]{xy}
\usepackage[linesnumbered,ruled]{algorithm2e}
\usepackage{setspace}
\usepackage{float}
\usepackage{amsfonts, amssymb}
\usepackage{amssymb,epsfig,psfrag}
\usepackage{here}
\usepackage{pict2e}
\usepackage{hyperref}
\usepackage{setspace}
\usepackage{url}
\usepackage{here}
\usepackage{pict2e}
\usepackage{hyperref}
\usepackage{blindtext}
\def\multiset#1#2{\ensuremath{\left(\kern-.3em\left(\genfrac{}{}{0pt}{}{#1}{#2}\right)\kern-.3em\right)}}

\hypersetup{
    colorlinks=true,
    linkcolor=blue,
    filecolor=magenta,      
    urlcolor=cyan,
}

\usepackage[letterpaper, left=2.5cm, right=2.5cm, top=2.5cm,
bottom=2.5cm,dvips]{geometry}

\newtheorem{theorem}{Theorem}
\newtheorem{lemma}[theorem]{Lemma}
\newtheorem{corollary}[theorem]{Corollary}

\theoremstyle{definition}
\newtheorem{definition}[theorem]{Definition}

\theoremstyle{remark}
\newtheorem{remark}[theorem]{Remark}

\title[Graphs associated with orthgonal collections of $k$-planes over finite fields]
{Graphs associated with orthogonal collections \\of $k$-planes over finite fields}

\author{Semin Yoo}

\date{\today}
\address{Department of Mathematics, University of Rochester, Rochester, NY, USA}

\email{syoo19@ur.rochester.edu}


\begin{document}

\maketitle

\begin{abstract}
We study graphs coming from quadratic spaces over finite fields via orthogonality which generalize a recent result given by Bishnoi, Ihringer, and Pepe (2019). More precisely, we study the graph $\Gamma^{\square}(n,k,q)$ as follows: the vertex set is the set of $k$-dimensional quadratic subspaces of a fixed Lorentzian quadratic space $(\mathbb{F}_{q}^{n},x_{1}^{2}+\cdots+x_{n-1}^{2}+\lambda x_{n}^{2})$ which are isometrically isomorphic to $x_{1}^{2}+\cdots+x_{k}^{2}$. Here $\lambda$ is a nonsquare in $\mathbb{F}_{q}$, and two vertices $x,y$ are adjacent if $x \subseteq y^{\perp}$. 
\end{abstract}

\section{Introduction and statements of results}

In \cite{BIP}, Bishnoi, Ihringer and Pepe construct clique-free pseudorandom graphs coming from finite geometry. Main ingredient is to use the theory of quadratic forms over finite fields. A history of their use can be found in the section $1$ of \cite{BIP}. In this paper, we generalize their results by changing their vertex set from $1$-dimensional quadratic subspaces of the Euclidean type to $k$-dimensional quadratic subspaces of the Euclidean type inside of the quadratic space $(\mathbb{F}_{q}^{n},\lambda \text{dot}_{n})$, where $\lambda$dot$_{n}=x_{1}^{2}+\cdots+x_{n-1}^{2}+\lambda x_{n}^{2}$, $\lambda$ is a nonsquare in $\mathbb{F}_{q}$ and $q$ is an odd prime power. 

\smallskip
The main goal of this paper is to study the properties of the graphs listed in the last column of Table \ref{table}.
\begin{table}[H]\label{table}
\begin{center}
\begin{tabular}{c||c|c}
 &  \textbf{The results in \cite{BIP} } & \textbf{The results in this paper} \\ \hline \hline
\textbf{ambient space} &  $(\mathbb{F}^{n}_{q},\lambda \text{dot}_{n})$ & 
$(\mathbb{F}^{n}_{q},\lambda \text{dot}_{n})$ \\ \hline
\textbf{vertex set} &  spacelike lines of $\mathbb{F}^{n}_{q}$ & 
dot$_{k}$-subspaces of $\lambda$dot$_{n}$ \\ \hline
\textbf{number of vertices} &  $(1+o(1))q^{n-1}/2$ & 
$(1+o(1))q^{k(n-k)}/2$ \\ \hline
\textbf{adjacency relation}   & $x \sim y ~ \Leftrightarrow ~ x \subseteq y^{\perp}$  & $x \sim y ~ \Leftrightarrow ~ x \subseteq y^{\perp}$ \\ \hline
\textbf{graph} & $\Gamma^{\square}(n,q)$ & $\Gamma^{\square}(n,k,q)$ \\ \hline
\multirow{4}{*}{\textbf{properties of the graph}} & (1) vertex-transitive & (1) vertex and arc-transitive  \\ 
& (2) $K_{2}$-free for any $n \geq 2$ & (2) $K_{2}$-free for any $k \geq n/2$ \\
 & (3) $K_{n}$-free for all $n \geq 2$   &  (3) $K_{l}$-free for all $l > \left [ \frac{n-1}{k} \right ]$\\
 & (4) $(n',d',\lambda')$-graph & (4) $(n'',d'',\lambda'')$-graph \\ 
\end{tabular}
\caption{The comparison between the results in \cite{BIP} and this paper.}
\label{comp}
\end{center}
\end{table}
\noindent where $n'=\Theta(q^{n-1}),d'=\Theta(q^{n-2}),\lambda' =\Theta(q^{(n-2)/2}),n''=\Theta(q^{k(n-k)}),d''=\Theta(q^{k(n-2k)})$ and $\lambda''=\Theta(q^{k(n-2k)/2})$.

\smallskip 
Furthermore, as an application, we study the orthogonality in the collection of $k$-dimensional quadratic subspaces the Euclidean type of the Lorentzian space $(\mathbb{F}_{q}^{n},\lambda \text{dot}_{n})$. Let us call a $k$-dimensional quadratic subspace $W$ of $(\mathbb{F}_{q}^{n},\lambda \text{dot}_{n})$, a $\mathbf{\textbf{dot}_{k}}\textbf{\text{-subspace}}$ if $W$ is isometrically isomorphic to $x_{1}^{2}+\cdots+x_{k}^{2}$ with the restricted quadratic form $\lambda$dot$_{n}|_{W}$. The number of dot$_{k}$-subspaces has an interesting combinatorial description. To know more about it, see \cite{Yo}. Given two collections of dot$_{k}$-subspaces $X,Y$, we find a necessary condition that ensures there exist dot$_{k}$-subspaces in the collections $X$ and $Y$ which are orthogonal to each other.   

\smallskip
\subsection*{Acknowledgements.} 
The author would like to express gratitude to Jonathan Pakianathan for helpful discussions and encouragement for this work. Also, the author would like to thank Ferdinand Ihringer for valuable comments for the revision.

\smallskip
\section{Generalization}

\smallskip
We briefly review some facts about quadratic forms over finite fields and mainly follow \cite{Yo}. Recall that any nondegenerate quadratic 
form on a vector space of dimension $n$ over a finite field is equivalent to one of the following two:
\begin{align*}
\text{dot}_n(\mathbf{x}) & :=x_{1}^{2}+\cdots+x_{n-1}^{2}+x_{n}^{2} \quad \text{or}\\
\text{$\lambda$dot$_n$}(\mathbf{x}) & :=x_{1}^{2}+\cdots+x_{n-1}^{2}+\lambda x_{n}^{2} \quad \text{for some nonsqaure $\lambda$ in $\mathbb{F}_{q}$}.
\end{align*}
Following physics notation, we will call dot$_{n}$ the \textbf{Euclidean form} and call $\lambda$dot$_{n}$ the \textbf{Lorentzian form}. When there is no danger of confusion, we will let dot$_{n}$ and $\lambda$dot$_{n}$ denote the quadratic spaces $(\mathbb{F}_q^n, \text{dot}_n)$ and $(\mathbb{F}_q^n, \lambda$dot$_n)$ respectively.

\begin{remark}
Finite geometers use the following classification of nondegenerate quadratic forms on a $n$-dimensional vector space $V$ over $\mathbb{F}_{q}$ instead of the previous one. A quadratic form $Q$ is equivalent to one of the following forms:
\begin{itemize}
\item (1) hyperbolic: $k\mathbb{H}$ when $n=2k$ and here $\mathbb{H}$ is the hyperbolic space,
\item (2) parabolic: $(k-1)\mathbb{H} \oplus (x^{2}-\lambda y^{2})$ when $n=2k$ and here $\lambda$ is nonsquare, 
\item (3) elliptic: $k\mathbb{H}\oplus cx^{2}$ when $n=2k+1$ and where $c$ is either $1$ or a nonsquare. 
\end{itemize}
One can correspond the previous classification to one of these three. For example, the elliptic case can be considered as both the Euclidean and the Lorentzian type by comparing their discriminants. 
\begin{itemize}
\item If $n \equiv 1$ (mod $4$), then $\mathbb{H}\oplus \cdots \oplus \mathbb{H} \oplus 1$ is equivalent to dot$_{n}$,
\item If $n \equiv 1$ (mod $4$), then $\mathbb{H}\oplus \cdots \oplus \mathbb{H} \oplus \lambda $ is equivalent to $\lambda$dot$_{n}$.
\item If $n \equiv 3$ (mod $4$), then $\mathbb{H}\oplus \cdots \oplus \mathbb{H} \oplus 1$ is equivalent to $\lambda$dot$_{n}$,
\item If $n \equiv 3$ (mod $4$), then $\mathbb{H}\oplus \cdots \oplus \mathbb{H} \oplus \lambda $ is equivalent to dot$_{n}$.
\end{itemize}
\end{remark}
\smallskip
In particular, there are three possible $1$-dimensional 
quadratic subspaces in $(\mathbb{F}^{n}_{q},\lambda\text{dot}_{n})$ up to equivalence:
(1) $\text{dot}_{1}$, (2) $\lambda\text{dot}_{1}$, and the degenerate case (3) $0$. Following physics notation again, we define three types of lines in quadratic space. This concept used in \cite{BIP} and \cite{Yo}.

\begin{definition}\cite{Yo}
The type of a line $l$ through the origin in $(\mathbb{F}_{q}^{n},\lambda\text{dot}_{n})$ is \textbf{spacelike} if $|l|$ is a square, \textbf{timelike} if $|l|$ is a nonsquare, and \textbf{lightlike} if $|l|$ is $0$. Here, we let $|l|$ denote the value $\lambda \text{dot}_{n}(\mathbf{x})$ for any nonzero $\mathbf{x}$ in $l$. 
\end{definition}


\smallskip 
It is easy to verify that the concept of line type is well-defined. We now define two graphs $\Gamma^{\square}(n,q)$ and $\Gamma^{\square}(n,k,q)$. In \cite{BIP}, the vertex set consists of spacelike lines of $\lambda$dot$_{n}$ denoted by $X_{\square}$ and two vertices $x$ and $y$ are adjacent if $x \subseteq y^{\perp}$. We denote the graph coming from this relation by $\Gamma^{\square}(n,q)$. Moreover, the size of the vertex set can be shown to be
\[(1+o(1))\frac{q^{n-1}}{2},\]
where $o(1)$ goes to zero as $q$ goes to infinity. The proof can be found in \cite{BIP}.

\smallskip
In this paper, we define a vertex set $V$ to be dot$_{k}$-subspaces of $(\mathbb{F}_{q}^{n},\lambda\text{dot}_{n})$. Note that dot$_{1}$-subspaces of $\lambda$dot$_{n}$ are just spacelike lines of $\lambda$dot$_{n}$. Let us define the graph $\Gamma^{\square}(n,k,q)$ with vertex set $V$ and $x$ adjacent to $y$ if and only if $x \subseteq y^{\perp}$ for any $x,y$ in $V$. 

\smallskip 
We now use some counts coming from \cite{Yo} to compute the number of the vertices of $\Gamma^{\square}(n,k,q)$. By section 4 of \cite{Yo}, the number of dot$_{k}$-subspaces of $\lambda$dot$_{n}$ denoted by $\binom{\lambda n}{k}_{d}$ is given as follows:
\[(1+o(1))\frac{q^{k(n-k)}}{2}.\]
As $k=1$, it becomes the size of the vertex set $X_{\square}$ of $\Gamma^{\square}(n,q)$.

\smallskip
We rely on the following two equivalent fundamental facts in the theory of quadratic forms. \begin{theorem}[Witt's extension theorem]\label{we}\cite{Cl}
Let $X_{1},X_{2}$ be quadratic spaces and $X_{1} \cong X_{2}$, $X_{1}=U_{1} \oplus V_{1},X_{2}=U_{2}\oplus V_{2}$, and $f:V_{1} \longrightarrow V_{2}$ an isometry. Then there is an isomtery $F:X_{1}\longrightarrow X_{2}$ such that $F|_{V_{1}}=f$ and $F(U_{1})=U_{2}$.
\end{theorem}

\begin{theorem}[Witt's cancellation theorem]\label{wc}\cite{Cl}
Let $U_{1},U_{2},V_{1},V_{2}$ be quadratic spaces where $V_{1}$ and $V_{2}$ are isometrically isomorphic. If $U_{1}\oplus V_{1} \cong U_{2} \oplus V_{2}$, then $U_{1} \cong U_{2}$.
\end{theorem}

Notice that the graph $\Gamma^{\square}(n,k,q)$ is a empty graph if $k \geq n$. Thus we assume that $k<n$.
\begin{lemma}
$\Gamma^{\square}(n,k,q)$ is $K_{2}$-free for any $k \geq n/2$.
\end{lemma}

\begin{proof}
Let $x$ be a vertex of $\Gamma^{\square}(n,k,q)$. Then there is a $k$-dimensional subspace $y$ such that $x \oplus y=\lambda \text{dot}_{n}$. Notice that we can also decompose $\lambda$dot$_{n}=x \oplus \lambda$dot$_{n-k}$. By Witt's cancellation theroem, we have $y=\lambda $dot$_{n-k}$ by cancelling $x$. To be a $K_{2}$-free graph, $n-k \leq k$ have to be hold. Thus $\Gamma^{\square}(n,k,q)$ is an edgeless graph when $n/2 \leq k$. 
\end{proof}

\begin{theorem}
For any $l > \left \lfloor \frac{n-1}{k}  \right \rfloor $, $\Gamma^{\square}(n,k,q)$ is $K_{l}$-free. Here, $\left \lfloor x  \right \rfloor$ is the largest integer less than or equal $x$.
\end{theorem}

\begin{proof}
In the graph $\Gamma^{\square}(n,k,q)$, $l$-clique is a $l$ times of direct sums of dot$_{k}$-subspaces which are orthogonal each other. In other words, it is dot$_{kl}$-subspaces. Thus the existence of $l$-clique is equivalent to have the condition $kl\leq n-1$ since dot$_{kl}\subseteq \lambda \text{dot}_{n}$. This is also equivalent to $l \leq \left \lfloor \frac{n-1}{k}  \right \rfloor$.
\end{proof}

\begin{lemma} Let $k < n/2$. $\Gamma^{\square}(n,k,q)$ is vertex-transitive and arc-transitive.
\end{lemma}

\begin{proof}
By Witt's extension theorem, the orthogonal group of $\lambda$dot$_{n}$ acts transitively on the vertex set. Thus $\Gamma^{\square}(n,k,q)$ is vertex-transitive. An $1$-arc $(X,Y)$ consists of a pair of orthogonal dot$_{k}$-subspaces. Since the orthogonal group $\lambda$dot$_{n}$ acts transitively on $(X,X^{\perp})$, it follows that $\Gamma^{\square}(n,k,q)$ is arc-transitive.
\end{proof}
By algebraic graph theory, vertex transitive graphs implies regular. Hence $\Gamma^{\square}(n,k,q)$ is regular.
\begin{lemma}
Let $k < n/2$. The induced graph on the neighborhood of any vertex of $\Gamma^{\square}(n,k,q)$ is isomorphic to $\Gamma^{\square}(n-k,k,q)$.
\end{lemma}

\begin{proof}
Since $\Gamma^{\square}(n,k,q)$ is vertex-transitive, without loss of generality, we can choose the vertex $x$ spanned by $e_{1},\cdots,e_{k}$, where $e_{i}$ are the standard orthonormal vectors. Then it is easy to check that $x$ is a dot$_{k}$-subspace. Since the ambient space is $\lambda$dot$_{n}$, $x^{\perp}$ is a $\lambda$dot$_{n-k}$ subspace. Therefore, the neighborhood of $x$ is isomorphic to $\Gamma^{\square}(n-k,k,q)$ if $k < n/2$.
\end{proof}
Thus $\Gamma^{\square}(n,k,q)$ is $d$-regular with 
\[d=\binom{\lambda (n-k)}{k}_{d}=(1+o(1))\frac{q^{k(n-2k)}}{2}.\]
To check $K_{2}$-freeness of $\Gamma^{\square}(n,k,q)$, we again rely on the fundamental fact in the theory of quadratic forms which is equivalent to Witt's extension theorem.

\smallskip
A $\bm{(n,d,\lambda)}$-\textbf{graph} is a $d$-regular graph with $n$ vertices and its second largest eigenvalue (in absolute value) at most $\lambda$. We already know that $\Gamma^{\square}(n,k,q)$ is $(1+o(1))q^{k(n-2k)}/2$-regular with $(1+o(1))q^{k(n-k)}/2$ vertices. In order to bound its second largest eigenvalue, we use interlacing of eigenvalues as \cite{BIP} does. See \cite{BIP} for the statement of interlacing. We first define a new graph $\overline{\Gamma}$ with the set of $k$-dimensional vector subspaces of $(\mathbb{F}_{q}^{n},\lambda\text{dot}_{n})$ as the vertex set and the adjacency is determined by orthogonality as before. We borrow some techniques from section 2 of \cite{AK} to prove the following theorem.

\begin{theorem}
Let $k < n/3$. The second largest eigenvalue $\overline{\lambda} $ of $\overline{\Gamma}$ is $(1+o(1))q^{k(n-2k)/2}$. 
\end{theorem}

\begin{proof}
We first note that the size of vertices of $\overline{\Gamma}$ is  $N=\binom{n}{k}_{q}=(1+o(1))q^{k(n-k)}$. From the properties of adjacency matrix $A$ of $\Gamma$, we obtain
\[A^{2}=AA^{T}=aJ+(d-a)I,\]
where $J$ is all $1$'s matrix, $I$ is the identity matrix, $a=\binom{n-2k}{k}_{q}=(1+o(1))q^{k(n-3k)}$ and $d=\binom{n-k}{k}_{q}=(1+o(1))q^{k(n-2k)}$. Here, $a$ is the number of $k$-dimensional subspaces $x$ such that $x \subseteq y^{\perp}$ and $x \subseteq z^{\perp}$ for given two distinct $k$-dimensional subspaces $y,z$. Similarly, $d$ is the number of $k$-dimensional subspaces $x$ satisfying $x \subseteq y^{\perp }$ for a given $k$-dimensional subspace $y$. Since $J$ and $I$ commute, $J$ and $I$ are simultaneously diagonalizable. Moreover, since $A$ is $d$-regular, the biggest eigenvalue of $A$ is $d$. Thus we have the following table.
\begin{table}[H]
\begin{center}
\begin{tabular}{c|c|c}
\textbf{eigenvalues of} $\mathbf{J}$ &  \textbf{eigenvalues of} $\mathbf{I}$ & \textbf{eigenvalues of} $\mathbf{A^{2}}$ \\ \hline \hline
$N$ &  $1$ & 
$aN+(d-a)1=d^{2}$ \\ \hline
$0$ &  $1$ & 
$d-a$ \\ \hline
$\vdots$ &  $\vdots$ & 
$\vdots$ \\ \hline
$0$   & $1$  & $d-a$ \\ 
\end{tabular}
\caption{Eigenvalues of $J,I$ and $A^{2}$.}
\label{comp}
\end{center}
\end{table}

\smallskip
Let $\overline{\lambda}_{1} \geq \overline{\lambda}_{2} \geq \cdots \geq  \overline{\lambda}_{n}$ be the eigenvalue of $\overline{\Gamma}$. Therefore, $\overline{\Gamma}$ has the second largest eigenvalue $\overline{\lambda}=\text{max}\left \{ \overline{\lambda}_{2},-\overline{\lambda}_{N} \right \}=\sqrt{d-a}=(1+o(1))q^{k(n-2k)/2}$.
\end{proof}

\begin{corollary}
Let $k < n/3$. Then $\overline{\Gamma}$ is an $(\overline{n},\overline{d},\overline{\lambda})$-graph with $\overline{n}=(1+o(1))q^{k(n-k)}, \overline{d}=(1+o(1))q^{k(n-2k)}$ and $\overline{\lambda}=(1+o(1))q^{k(n-2k)/2}$.
\end{corollary}

\begin{corollary}\label{cor9}
Let $k < n/3$. Then $\Gamma^{\square}(n,k,q)$ is an $(n'',d'',\lambda'')$-graph with $n''=(1+o(1))q^{k(n-k)}/2,d''=(1+o(1))q^{k(n-2k)}/2$ and $\lambda''=(1+o(1))q^{k(n-2k)/2}$.
\end{corollary}

\begin{proof}
Let $\lambda_{1}'' \geq \lambda_{2}'' \geq \cdots \geq \lambda_{n''}''$ be the eigenvalues of $\Gamma^{\square}(n,k,q)$. By interlacing, $\overline{\lambda}_{i} \geq \lambda_{i}'' \geq \overline{\lambda}_{N-n''+i}$ for any $i=1,\cdots,n$. It follows that $\lambda''_{2} \leq \overline{\lambda}_{2}$ and $-\lambda''_{n} \leq -\overline{\lambda}_{N}$. Therefore, we have $\lambda''=\text{max}\left \{ \lambda_{2}'',-\lambda_{n}'' \right \}\leq \overline{\lambda}=(1+o(1))q^{k(n-2k)/2}$.
\end{proof}

As an application, we apply spectral gap theorem to the graph $\Gamma^{\square}(n,k,q)$. The spectral gap theorem says that for a finite $d$-regular graph with a vertex set $V$, if $X,Y \subseteq V$ and
\[\sqrt{|X||Y|}>n_{*}=|V|\left (\frac{\text{max}(|\lambda_{2}|,\cdots,|\lambda_{n}|)}{d}\right ),\]
then there exist edges between sets $X$ and $Y$. For the graph $\Gamma^{\square}(n,k,q)$, let $V(\Gamma^{\square}(n,k,q))$ be the vertex set of $\Gamma^{\square}(n,k,q)$ and $X,Y$ be subsets of $V(\Gamma^{\square}(n,k,q))$. By Corollary \ref{cor9}, we obtain
\[n_{*}=n''\left(\frac{\text{max}(\lambda_{2}'',-\lambda_{n''}'')}{d''} \right)\leq (1+o(1))\frac{q^{k(n-k)}}{2}\left ( \frac{(1+o(1))q^{k(n-2k)/2}}{(1+o(1))q^{k(n-2k)}/2} \right )\approx (1+o(1))q^{\frac{nk}{2}}.\]
Therefore, by spectral gap theorem, we have:

\begin{corollary}
Let $k < n/3$. If $|X||Y|>(1+o(1))q^{nk}$, then there exist edges between $X$ and $Y$. In other words, there exists dot$_{k}$-subspaces in $X$ and $Y$ which are orthogonal each other.
\end{corollary}


\medskip
\begin{thebibliography}{1000}

\bibitem{AK} Noga Alon and M. Krivelevich, \textit{Constructive Bounds for a Ramey-Type Problem}, Graph and Combinatorics {\bf 13} (1997), 217-225.


\bibitem{BIP} A. Bishnoi, F. Inhringer and V. Pepe, \textit{Construction for clique-free pseudorandom graphs}, Combinatorica, to appear.

\bibitem{Cl} P.L. Clark, \textit{Quadratic forms chapter I: Witt's theory}, \url{http://math.uga.edu/~pete/quadraticforms.pdf}

\bibitem{Co} K. Conard, \textit{Bilinear Forms}, \url{http://www.math.uconn.edu/~kconrad/blurbs/linmultialg/bilinearform.pdf}

\bibitem{GR} C Godsil and G. Royle, \textit{Algebraic graph theory}, Springer (2001).

\bibitem{Yo} S. Yoo, \textit{Combinatorics of quadratic spaces over finite fields}, preprint (2019).



\end{thebibliography}
\end{document}